\documentclass[12pt]{amsart}
\usepackage{amssymb,amsmath,mathrsfs,setspace,amsthm}
\usepackage{color}
\usepackage{ytableau}
\usepackage{fullpage}
\theoremstyle{plain}
\newtheorem{theorem}{Theorem}[section]
\newtheorem{lemma}[theorem]{Lemma}
\newtheorem{corollary}[theorem]{Corollary}
\newtheorem{proposition}[theorem]{Proposition}

\theoremstyle{definition}

\theoremstyle{remark}

\newcommand{\RR}{\mathbb{R}}

\title{Sums of squares on the hypercube}
\author{Grigoriy Blekherman}
\address{Grigoriy Blekherman,
School of Mathematics,
Georgia Institute of Technology,
686 Cherry Street,
Atlanta, GA 30332-0160 USA}
\email{greg@math.gatech.edu}
\author{Jo{\~a}o Gouveia}
\address{Jo{\~a}o Gouveia, CMUC, Department of Mathematics,
  University of Coimbra, 3001-454 Coimbra, Portugal}
\email{jgouveia@mat.uc.pt} 
\author{James Pfeiffer}
\address{James Pfeiffer, Department of Mathematics,
University of Washington, Seattle, WA 98195}
\email{jamesrpfeiffer@gmail.com}
\thanks{The authors were partially supported on this project 
as follows: GB by the Sloan Research Fellowship, JG by `Centro de Matem\'{a}tica da Universidade de Coimbra' and `Funda\c{c}\~{a}o para a Ci\^{e}ncia e a Tecnologia', through European program COMPETE/FEDER; and JP by NSF grant DMS-1115293}

\begin{document}

\begin{abstract}
Let $X$ be a finite set of points in $\RR^n$. A polynomial $p$ nonnegative on $X$ can be written as a sum of squares of rational functions modulo the vanishing ideal $I(X)$. From the point of view of applications, such as polynomial optimization, we are interested in rational function representations \emph{of small degree}. We derive a general upper bound in terms of the Hilbert function of $X$, and we show that this upper bound is tight for the case of quadratic functions on the hypercube $C=\{0,1\}^n$, a very well studied case in combinatorial optimization. Using the lower bounds for $C$ we construct a family of globally nonnegative quartic polynomials, which are not sums of squares of rational functions of small degree. To our knowledge this is the first construction for Hilbert's 17th problem of a family of polynomials of bounded degree which need increasing degrees in rational function representations as the number of variables $n$ goes to infinity. We note that representation theory of the symmetric group $S_n$ plays a crucial role in our proofs of the lower bounds.
\end{abstract}

\maketitle

\section{Introduction}\label{Section:Introduction}

Certifying that a polynomial $p$ is nonnegative on a finite set $X$ in $\RR^n$ is an important problem in optimization, as certificates of nonnegativity can often be leveraged into optimization algorithms. One frequently used certificate is writing $p$ as a sum of squares of polynomials modulo the vanishing ideal $I(X)$ of $X$. These certificates lead to semidefinite relaxations for the problem of optimizing a polynomial on $X$ \cite{lasserre, parrilo2}. For instance, when $X$ is the hypercube $\{0,1\}^n$%and $q$ is a quadratic function,
, maximizing a quadratic polynomial on $X$ specializes to many famous combinatorial optimization problems such as MAXCUT. Sums of squares certificates provide a way of  automatically constructing semidefinite relaxations for these problems. The celebrated Goemans-Williamson relaxation algorithm, for instance, can be seen as such a sum of squares relaxation \cite[Chapter 2 and 3]{BPT}, \cite{goemans_williamson, marshall2008}.

In general, one might be required to use polynomials of high degree to certify that $p$ is nonnegative on $X$. Since Hilbert's 17th problem, it is classical in real algebraic geometry to certify nonnegativity of a polynomial by writing it as a sum of squares of rational functions, instead of polynomials. This can be reformulated as follows:
$$\text{Given $p$ find a sum of squares $h$, such that $ph$ is a sum of squares modulo $I(X)$}.$$

When $X=\RR^n$, the existence of such certificates for any nonnegative polynomial corresponds to Hilbert's 17th problem, and was answered affirmatively by Artin. For a general semialgebraic set the existence of such certificates is guaranteed by Stengle's Positivstllensatz, which was later refined by Schm\"{u}dgen, Putinar and Jacobi. See for example \cite{marshall2008, prestel2001} for an in-depth discussion of these topics.
We are interested in showing \emph{degree bounds} on the degree of the multiplier $h$. There are known general upper bounds coming from real algebraic geometry for rational function certificates on any real semialgebraic set $X$ \cite{Lombardi, LPR, Schmid}. %\jg{reference}. 
However, they result in bounds which are multiple towers of exponentials. We are not aware of any general lower degree bounds, even for Hilbert's 17th problem. For some specific small cases see \cite{GKZ}.

For the case when $X$ is a finite set of points, one of our main results is an elementary uniform upper bound on the degree of the multiplier $h$, in terms of the Hilbert function of $X$ and the degree of $p$. Our second main result is showing this bound is tight for the case of quadratic functions on the hypercube $C=\{0,1\}^n$. We leverage the tightness of the bound on $C$ into a construction of a globally nonnegative polynomial $p$ of degree $4$ in $n$ variables such that $ph$ is not a sum of squares for all sums of squares $h$ of degree at most $2\lfloor n/2 \rfloor -4$. While this bound can very likely be improved, to our knowledge this is the first construction for Hilbert's 17th problem of a polynomial of bounded degree, which needs multipliers $h$ of increasing degree as the number of variables $n$ goes to infinity.

\subsection{Background, Discussion and Main Results.} Let $X\subset \RR^n$ be a real variety and let $I=I(X)$ be its vanishing ideal. Let $\RR[X]=\RR[x_1,\dots,x_n]/I$ be the coordinate ring of $X$. Given $f \in \RR[X]$ we define \emph{degree} of $f$ as the lowest degree of any polynomial in the equivalence class $f +I$. Let $\RR[X]_{\leq d}$ be the real vector space of polynomials of degree at most $d$ in $\RR[X]$. Recall that the Hilbert function $H_X(t)$ of $X$ is defined as follows:
$$H_X(t)=\dim \RR[X]_{\leq t}.$$
We say that $f \in \RR[X]$ is $k$-sos  if there exist $g_1,\dots, g_m \in \RR[X]_{\leq k}$ such that $f=g_1^2+\dots+g_m^2$. The set of all $k$-sos polynomials will be denoted by
$\Sigma(X)_{\leq 2k}$. This set of polynomials has attracted strong attention from the optimization community in recent years, as a relaxation for the cone of  polynomials nonnegative on $X$ \cite{GLPT, gpt,lasserre2002,laurent2007}. The reason for this is that checking whether a polynomial is $k$-sos is a semidefinite feasibility problem and, even better, one can use semidefinite programming to optimize a linear functional over the cone of $k$-sos polynomials \cite[Chapters 2 and 6]{BPT}. 

For a compact variety $X$, Schm\"{u}dgen's Positivstellensatz implies that any polynomial that is strictly positive on $X$ is $k$-sos for large enough $k$. However there may be no uniform bounds on this $k$ for all polynomials of fixed degree. %\gb{add some previous results on k-sos}
This situation improves considerably if we allow sums of squares of rational functions. We say that $p \in \RR[X]_{\leq 2s}$ is $(d,k)$-rsos (rational sum of squares) if there exists non-zero $h\in \Sigma(X)_{\leq 2d}$ such that $ph \in \Sigma(X)_{\leq 2k}$. We will omit $d$ and write simply that $p$ is $k$-rsos for the case $d = k-s$. It follows from Stengle's Positivstellensatz that  for any polynomial $p$ nonnegative on $X$ there is a $k\in \mathbb{N}$ for which $f$ is $k$-rsos. 
Moreover, there is a bound on $k$ that depends only on the degree of $p$ and the variety $X$. 
The trade-off is that, computationally, this certificate has worse properties: while checking if a polynomial is $k$-rsos is still a semidefinite feasibility problem, the set of all such polynomials has no direct semidefinite description, and tools other than semidefinite programming have to be used to optimize over it. Moreover, when $X$ is a \emph{reducible} variety, a non-zero sum of squares multiplier $h$ such that $ph$ is a sum of squares is not necessarily a certificate of nonnegativity of $p$. This happens since $h$ may vanish identically on a component of $X$, and on this component nonnegativity of $p$ is not certified. Therefore, %to certify nonnegativity 
we will also be interested in the existence of \emph{strictly positive} sum of squares multipliers $h$.

In the case $X$ is a finite set of points in $\RR^n$, there exist uniform degree bounds for $k$-sos representations. The \emph{Hilbert regularity} $h(X)$ of $X$ is the smallest degree $d$ for which $H_X(d)=|X|$ and, consequently, $H_X(t)=|X|$ for all 
$t \geq h(X)$. A polynomial $f \in \RR[X]$ is uniquely determined by its values on $X$, so we may identify elements of $\RR[X]$ with functions on $X$. For a point $v \in X$ let 
$\delta_v:X \rightarrow \RR$ be the interpolator of $v$:
$\delta_v(v)=1$ and $\delta_v(x)=0,\,\,\,\, x \neq v.$  We note that $h(X)$ is the smallest degree $d$ such $\delta_v \in \RR[X]_{\leq d}$ for all $v \in X$. 
Furthermore, using interpolators we can write any $p \in \RR[X]$ as: $$p=\sum_{v \in X}p(v)\delta_v^2.$$ 
It follows that any nonnegative polynomial $p \in \RR[X]$ is $h(X)$-sos. It is not difficult to construct examples of finite sets $X$ and nonnegative polynomials $p \in \RR[X]$ of any degree, such that $p$ is not $(h(X)-1)$-sos, i.e. we may need to go all the way up to Hilbert regularity to certify nonnegativity of $p$. 

For the rational function representations we provide better upper bounds by using the following result.
\begin{theorem}\label{Theorem:MainBound} Let $X$ be a finite set of points in $\RR^n$.
Let $p \in \RR[X]_{\leq 2s}$ be a polynomial of degree at most $2s$ nonnegative on $X$. Suppose that for some $k \in \mathbb{N}$ we have
$$H_X(k+s)+H_X(k)>H_X(2k+2s).$$
Then $p$ is $(k+s)$-rsos on $X$, i.e. there exists $h \in \Sigma(X)_{\leq 2k}$ such that $ph \in \Sigma(X)_{\leq 2s+2k}$.
\end{theorem}
%This essentially tells us that if for some $d$, $H_X(2d)$ is ``significantly'' smaller than $2H_X(d)$, $d$-sos polynomials can capture the nonnegativity of small degree polynomials, when denominators are allowed. For varieties where after some point the Hilbert function grows sufficiently slowly, we can therefore extract some non-trivial bounds. \gb{The result only works for $0$-dimensional varieties, so I don't think the above discussion works}
An important application of the above theorem is to quadratic polynomials on the hypercube $C=\{0,1\}^n$. It is easy to show that $H_C(t)=\sum_{i=0}^t \binom{n}{i}$ and therefore $H_C(n)=2^n=|C|$, while $H_C(\lfloor \frac{n}{2}\rfloor+1)+H_C(\lfloor \frac{n}{2}\rfloor)>2^n$. This implies that all nonnegative quadratic polynomials on the hypercube are $(\lfloor \frac{n}{2}\rfloor+1)$-rsos. In fact this result is tight since we also show the following:
\begin{theorem}\label{Corollary:Quadratic}
Let $k=\lfloor \frac{n}{2}\rfloor$ and let $f\in \RR[C]$ be given by $$f=(x_1+\dots+x_n-k)(x_1+\dots+x_n-k-1).$$
Then $f$ is nonnegative on $C$ but $f$ is not $k$-rsos.
\end{theorem}

Our proof relies on symmetries of the polynomial $(x_1+\dots+x_n-k)(x_1+\dots+x_n-k-1)$ and we use representation theory of the symmetric group $S_n$ in an essential way. More general lower bounds for rational function representations of symmetric polynomials on the hypercube are given in Theorem \ref{Theorem:SosMultipliers} and Theorem \ref{Corollary:Quadratic} is a direct corollary. 

From Theorem \ref{Corollary:Quadratic} we can derive two interesting results. First, it immediately recovers a result by Laurent \cite{moniquestuff} concerning the power of $k$-sos representations 
for relaxations of the MAXCUT problem. In fact, we significantly strengthen that result by proving that it remains true even for rational sums of squares representations, and by proving 
that in this case, the bounds are optimal. 

If we demand the certificates to be strictly positive, the case most pertinent to optimization, we prove in Theorem \ref{Theorem:QuadraticBoundPositive} that for the case of quadratic functions on the hypercube $C$ the bound of Theorem \ref{Theorem:MainBound} needs to be increased by at most $1$ degree, and thus it is still almost optimal.
%is still almost optimal, needing at most one more degree than the one obtained for nonnegative certificates in Theorem \ref{Theorem:MainBound}.

We also use Theorem \ref{Corollary:Quadratic} to provide lower bounds for the degree of the denominators in Hilbert's 17th problem.
More precisely, we use the quadratic polynomial nonnegative on the hypercube to construct a family of globally nonnegative quartic polynomials in $n$ variables which are not 
$\lfloor \frac{n}{2}\rfloor$-rsos. This is, to our knowledge, the first example of a family of polynomials of bounded degree which needs denominators of increasing degree in their representations as sums of squares of rational functions.

\section{Upper Bound on Multipliers}\label{Section:UpperBound}

Let $X=\{v_1,\ldots,v_m\}$ be a finite set of points in $\RR^n$. %We consider sums of squares in $\RR[X]$. 
We first show that the set of $(d_1,d_2)$-rsos polynomials is always closed. %\gb{fix notation here}

\begin{lemma} \label{Lemma:GenClosed}
Fix $d_1,d_2\in \mathbb{N}$. The set of polynomials in $\RR[X]_{\leq 2d}$ which are $(d_1,d_2)$-rsos is closed for all $d_1$, $d_2$, and $d$.
\end{lemma}
\begin{proof}
One can check that $\Sigma(X)_{\leq 2d}$ is a closed pointed convex cone in $\RR[X]_{\leq 2d}$  \cite[Chapter 4]{BPT}.
Suppose that $f_i \in \RR[X]_{\leq 2d}$ are $(d_1,d_2)$-rsos and converge to $f$. Then there exist $g_i$, $h_i$ which are respectively $d_1$ and $d_2$-sos and $f_ig_i=h_i$. We may rescale $g_i$ and assume that $$\frac{1}{m}\sum_{j=1}^m g_i(v_j)=1.$$
The set of $d_1$-sos polynomials with average $1$ on $X$ is compact. Therefore a subsequence of $\{g_i\}$ converges to $g$, which is also $d_1$-sos. Then the corresponding subsequence of $f_ig_i$ converges to $fg$ and, since each $f_ig_i$ is $d_2$-sos, it follows that $fg$ is $d_2$-sos. 

\end{proof}

We now develop some results about linear functionals on $\RR[X]_{\leq 2d}$ that are nonnegative on $k$-sos polynomials. These results are based on elementary dimension counting, but they will be crucial in the proof of Theorem \ref{Theorem:MainBound} as we will be able to conclude non-existence of a certain separating linear functional. Let $\ell: \RR[X]_{\leq 2d} \rightarrow \RR$ be a linear functional given as a combination of point evaluations on $X$:
$$\ell(f)=\sum_{i=1}^m \mu_if(v_i),\quad f \in \RR[X]_{\leq 2d}, \, \mu_i \in \RR.$$

\noindent We assume that the coefficients $\mu_i$ are non-zero and let $m_+$ and $m_-$ be the number of positive and negative $\mu_i$ respectively, and let $Q_{\ell}: \RR[X]_{\leq d} \rightarrow \RR$ be the quadratic form associated to $\ell$ given by 
$$Q_{\ell}(f)=\ell(f^2)=\sum_{i=1}^m \mu_if^2(v_i).$$ 
  
\begin{lemma}\label{Lemma:Signs} Let $\ell: \RR[X]_{\leq 2d} \rightarrow \RR$ be given by $\ell(f)=\sum_{i=1}^m \mu_if(v_i)$ with all $\mu_i \neq 0$. Suppose that 
%the quadratic form $Q_{\ell}$ is positive semidefinite. 
$\ell$ is nonnegative on $\Sigma(X)_{\leq 2d}$. 
Then $m_+\geq \dim \RR[X]_{\leq d}$.
\end{lemma}
\begin{proof}
Let $\pi_{X}: \RR[X]_{\leq d} \rightarrow \mathbb{R}^{m}$ be the evaluation projection of forms in $\RR[X]_{\leq d}$ given by
$$\pi_{X}(f)=\left(f(v_1),\ldots,f(v_m)\right), \quad f \in R[X]_{\leq d}.$$

\noindent We observe that the map $\pi_{X}$ has a trivial kernel and therefore $$\dim \pi_{X}(\RR[X]_{\leq d})=\dim \RR[X]_{\leq d}.$$

\noindent Let $\bar{Q}_{\ell}$ be the quadratic form on $\RR^{m}$ given by: $$\sum_{i=1}^m \mu_ix_i^2.$$

\noindent By its definition, the form $Q_{\ell}$ is a composition of $\pi_{X}$ and $\bar{Q}_{\ell}$:
$$Q_{\ell}=\bar{Q}_{\ell}\circ \pi_{X}.$$
The form $\bar{Q}_{\ell}$ has $m_-$ negative eigenvalues, and thus $\bar{Q}_{\ell}$ is strictly negative on a subspace of dimension $m_-$. Recall that the form $Q_{\ell}$ is positive semidefinite, which implies that $\bar{Q}_{\ell}$ is positive semidefinite on the image of $\pi_{X}$. Thus the image of $\pi_{X}$ has codimension at least $m_-$ in $\RR^m$. Since $m_++m_-=m$ the Lemma follows.
\end{proof}

%Let $H_X(t)$ denote the Hilbert function of $X$: $$H_X(t)=\dim \RR[X]_{\leq t}.$$

%\begin{theorem}\label{Theorem:MainBound}
%Let $p \in \RR[X]_{\leq 2s}$ be a polynomial of degree at most $2s$ nonnegative on $X$. Suppose that for some $k \in \RR$ we have
%$$H_X(k+s)+H_X(k)>H_X(2k+2s).$$
%Then $p$ is $(k,k+s)$-sos on $X$, i.e. there exists $q \in \Sigma(X)_{\leq 2k}$ such that $pq \in \Sigma(X)_{\leq 2s+2k}$.
%\end{theorem}

We are now in position to prove Theorem \ref{Theorem:MainBound}.
\begin{proof}[Proof of Theorem \ref{Theorem:MainBound}]
Suppose not. By Lemma \ref{Lemma:GenClosed}, the set of all polynomials in $\RR[X]_{\leq 2s}$ that is not $(k+s)$-rsos is open. Thus we can find $p\in \RR[X]_{\leq 2s}$ that is strictly positive on $X$ but is not $(k+s)$-rsos. Now consider the pointed, closed convex cones $p\Sigma(X)_{\leq 2k}$ and $\Sigma(X)_{\leq 2k+2s}$ in $\RR[X]_{\leq 2k+2s}$. By our assumption $$p\Sigma(X)_{\leq 2k} \cap \Sigma(X)_{\leq 2k+2s}=\{0\}.$$
Therefore there exists a linear functional $\ell:\RR[X]_{\leq 2k+2s} \rightarrow \RR$ strictly separating the two cones: $\ell(f)>0$ for all nonzero $f \in \Sigma(X)_{\leq 2k+2s}$ and $\ell(f)<0$ for all nonzero $f \in p\Sigma(X)_{\leq 2k}$. 

Let $X'\subseteq X$ be a subset of $X$ such that point evaluations on $X'$ form a basis of the dual space of linear functionals $\RR[X]_{\leq 2k+2s}^*$. We note that $$|X'|=\dim \RR[X]_{\leq 2k+2s}\,\,\,\, \text{and}\,\,\,\, \dim \RR[X']_{\leq d}=\RR[X]_{\leq d}\,\,\,\, \text{for all} \,\,\,\, d \leq 2k+2s.$$ Therefore the separating functional $\ell$ can be written as 
$$\ell=\sum_{v_i \in X'} \mu_i\ell_{v_i}, \,\,\,\, \mu_i \in \RR,$$
where $\ell_{v_i}$ are point evaluation functionals on points of $X'$. Let $p^\prime$ be the image of $p$ under the canonical projection from $\RR[X]$ to $\RR[X']=\RR[X]/I(X')$. It follows that $\ell$ also strictly separates $p'\Sigma_{\leq 2k}(X')$ from $\Sigma_{\leq 2k+2s}(X')$ and $p'$ is strictly positive on $X'$.
Since $\ell$ strictly separates the two cones we may assume without loss of generality that all coefficients $\mu_i$ are non-zero. 
Let $m_+$ and $m_-$ be the number of positive and negative $\mu_i$ respectively. Then by Lemma \ref{Lemma:Signs} we know that $m_+ \geq \dim \RR[X']_{\leq k+s}=\dim \RR[X]_{\leq k+s}$. 
% \jg{This is not quite direct from Lemma 2.2, since we should have a form in $\RR[X']_{\leq k+s}$ and not $\RR[X]_{\leq k+s}$. I guess the important thing is that $\pi_{X'}: \RR[X]_{\leq k+s} \rightarrow \RR^{|X'|}$ has trivial kernel, but maybe we should change Lemma 2.2 to reflect that}\gb{I rewrote the proof a bit to address the point that we are applying Lemma 2.2 to $X'$, not $X$. }

Now define $\ell': \RR[X']_{\leq 2k} \rightarrow \RR$ by $$\ell'=\sum_{v_i \in X'}\mu_ip'(v_i)\ell_{v_i}.$$
%It follows that $Q_{\ell'}:\RR[X']_{\leq k} \rightarrow \RR$ is a negative definite quadratic form. 
The functional $\ell'$ is nonnegative on $\Sigma_{\leq 2k}(X')$, therefore, by applying Lemma \ref{Lemma:Signs}, we see that $m_-\geq \dim \RR[X']_{\leq k}=\dim \RR[X]_{\leq k}$, since $p'(v_i)>0$ for all $v_i\in X'$. Combining, we see that $$H_X(2k+2s)=|X'|=m_+ + m_- \geq H_X(k+s)+H_X(k),$$
which is a contradiction.
\end{proof}

\begin{corollary}\label{Corollary:QuadraticBound}
Let $p\in \RR[C]_{\leq 2}$ be a quadratic polynomial nonnegative on $C$ and let $k=\lfloor \frac{n}{2}\rfloor$. Then $p$ is $(k+1)$-rsos.
\end{corollary}
\begin{proof}
This follows immediately from Theorem \ref{Theorem:MainBound} since $H_C(t)=\sum_{i=0}^t \binom{n}{i}$.
\end{proof}
%This gives insight into a conjecture of Laurent \cite{moniquestuff} that the Lasserre rank of the max cut polytope on $K_n$ is exactly $n/2$.
%It was already known to be at least $n/2$; see for instance Section \ref{Subsection:MaxCut} of this paper. 
%By Corollary \ref{Corollary:QuadraticBound}, rank $n/2$ is required when using multipliers.
%It remains to check whether allowing multipliers reduces the degree required in this case.

\subsection{Strictly Positive Multipliers}

We observe that having a $k$-rsos representation of a polynomial $p \in \RR[X]$ is not in general a certificate of nonnegativity of $p$. This is due to the fact that $X$ is a reducible variety and the multiplier $h$ may vanish on some points of $X$. On these points nonnegativity of $p$ is not certified. %In an extreme case we may find $h \in \Sigma_{2k}(X)$ such that $fh=0$, i.e. $fh \in I(X)$ and in this case 

Therefore we are interested in showing existence of strictly positive sum of squares multipliers. More specifically we will be interested in multipliers $h$ of the form
$$h=1+\sum q_i^2, \,\,\,\, q_i\in \RR[X]_{\leq k}.$$ 
We note that, up to multiplication by a positive constant, such sums of squares correspond precisely to the interior points of the cone $\Sigma(X)_{\leq 2k}$. %since the property of being a multiplier is preserved under multiplication by a positive constant. 
%Therefore, we are interested in proving nonexistence of \emph{weak separation} between the cones $\Sigma_{\leq 2s+2k}(X)$ and $p\Sigma_{\leq k}(X)$ in the same way as in the proof of the Theorem \ref{Theorem:MainBound}. 
We will concentrate on the case of a quadratic polynomial nonnegative on a subset $X$ of the hypercube $C$.  We first show that the bound of $d=\lfloor \frac{n}{2}\rfloor$ suffices also for any strictly positive quadric $p\in \RR[X]_{\leq 2}$. 

%Let $C$ be the unit cube given by $x_i^2-x_i=0$ for $i=1,\dots,n$. More generally $C$ can be any completely real transversal $0$-dimensional intersection of $n$ quadrics in $\RR^{n}$.

\begin{theorem}\label{THM pos}
Let $d=\lfloor \frac{n}{2} \rfloor$ and let $X$ be a subset of $C$. If $p \in \RR[X]_{\leq 2}$ is a quadratic polynomial that is strictly positive on $X$ then there exists $h$ in the interior of $\Sigma(X)_{\leq 2d}$ such that $p \cdot h$ lies in the interior of $\Sigma(X)_{\leq 2d+2}$.
\end{theorem}
\begin{proof}
Suppose not. Then the pointed convex cones $p\Sigma(X)_{\leq 2d}$ and $\Sigma(X)_{\leq 2d+2}$ can be weakly separated. Therefore there exists a linear functional $\ell \in \RR[X]_{\leq 2d+2}^*$ such that $\ell(s) \geq 0$ for all $s \in \Sigma(X)_{\leq 2d+2}$ and $\ell(s) \leq 0$ for all $s \in p\Sigma(X)_{\leq 2d}$. We can write
$$\ell=\sum_{v_i \in X} \mu_i\ell_{v_i}, \hspace{4mm} \mu_i \in \RR.$$
%We may assume without loss of generality that the point evaluation functionals $\ell_{v_i}$ are linearly independent. 
Let $X'$ be the subset of $X$ corresponding to non-zero coefficients $\mu_i$. Let $p^\prime$ be the image of $p$ under the canonical projection from $\RR[X]$ to $\RR[X']=\RR[X]/I(X')$. It follows that $\ell$ also separates $p'\Sigma(X')_{\leq 2d}$ from $\Sigma(X')_{\leq 2d+2}$ and $p'$ is strictly positive on $X'$. 

%We observe that by linear independence of the functionals $\ell_{v_i}$ it follows that:
 %$$|X'|=\dim \RR[X]_{\leq 2d+2}\,\,\,\, \text{and}\,\,\,\, \dim \RR[X']_{\leq k}=\RR[X]_{\leq k}\,\,\,\, \text{for all} \,\,\,\, k \leq 2d+2.$$
 
%\gb{What I meant here is that we will get the same contradiction from $X'$ as from $X$, so there is no real reason to use $X'$. However, it's not really possible to see this without reading the proof to the end... So I put in $X'$ everywhere below.}
%Therefore, without loss of generality we may assume that all $\mu_i$ are non-zero, otherwise we may switch and prove the equivalent statement for $p'$ and $X'$.
%\jg{I don't completely follow this reduction argument. Reducing to $X'$ does not prove the original statement of the Theorem, does it? I guess it would proof a slightly weaker version that is still strong enough for Theorem 2.5, but maybe I'm just really not seeing it...}

%INSERT USUAL COUNTING ARGUMENT HERE.

Let $m_+$ and $m_-$ be the number of positive and negative $\mu_i$ respectively. Using Lemma \ref{Lemma:Signs} we see that $m_+ \geq \dim \RR[X']_{\leq d+1}$. On the other hand we may define $\ell':\RR[X']_{\leq 2d}\rightarrow \RR$ by 
$$\ell'(q)=\ell(p'q), \hspace{1cm} \ell'=\sum_{v_i \in X'} \mu_ip'(v_i)\ell_{v_i}.$$
Since $p'$ is strictly positive on $X'$ and $\ell'$ is nonpositive on squares we can apply Lemma \ref{Lemma:Signs} to see that $m_- \geq \dim \RR[X']_{\leq d}$.

We now claim that \begin{equation}\label{EQN key}\dim \RR[X']_{\leq d}+\dim \RR[X']_{\leq d+1} > |X'|.\end{equation}
%We observe that since $2d+2>n$ it follows that $ \dim R[X]_{2d+2}=|X|$. 
Let $\bar{X}'$ denote the complement of $X'$ in $C$. Using Cayley-Bacharach duality \cite{CB}, we see that $|X'|-\dim \RR[X']_{\leq d}=\dim \RR[C]_{\leq n-d-1}-\dim \RR[\bar{X'}]_{\leq n-d-1}$. We observe that $d+1>n-d-1$ and we must have $\dim \RR[X']_{\leq d+1}>\dim \RR[X']_{\leq n-d-1}$, otherwise $\dim \RR[X']_{\leq d+1}=\dim \RR[X']_{\leq d}=|X'|$ and \eqref{EQN key} is proved. Thus we have 
\begin{align*}&\dim \RR[X']_{\leq d+1}+\dim \RR[X']_{\leq d} - |X'|=\dim \RR[X']_{\leq d+1}+\dim \RR[\bar{X'}]_{\leq n-d-1}-\dim \RR[C]_{\leq n-d-1}>\\
&\dim \RR[X']_{\leq n-d-1}+\dim \RR[\bar{X'}]_{\leq n-d-1}-\dim \RR[C]_{\leq n-d-1}\geq 0.\end{align*}
This finishes the proof of the claim, and now we observe that since $m_+ + m_-=|X'|$ we have reached a contradiction.
\end{proof}

We now show that if $p\in \RR[X]_{\leq 2}$, is nonnegative on $X \subseteq C$ then there are interior sum of squares multipliers of degree at most $\lfloor\frac{n}{2}\rfloor+1$, i.e. we may need to increase the degree by $1$ in order to certify nonnegativity of a quadric. It is not clear to us whether this is truly necessary, or perhaps there exist interior sum of squares multipliers of degree at most $\lfloor\frac{n}{2}\rfloor$.

\begin{theorem}\label{Theorem:QuadraticBoundPositive}
Let $d=\lfloor \frac{n}{2} \rfloor$ and let $X$ be a subset of $C$. If $p \in \RR[X]_{\leq 2}$ is a non-zero quadratic function nonnegative on $X$, then there exists $h$ in the interior of $\Sigma(X)_{\leq 2d+2}$ such that $p \cdot h \in \Sigma(X)_{\leq 2d+4}$.
\end{theorem}
\begin{proof}
It is equivalent to show that any linear functional in $\RR[X]^*_{\leq 2d+4}$ which separates $p\Sigma(X)_{\leq 2d+2}$ and $\Sigma(X)_{\leq 2d+4}$ is identically zero on $p\Sigma(X)_{\leq 2d+2}$. Let $\ell$ be such a functional. We can write
$$\ell=\sum_{v_i \in X}\mu_i \ell_{v_i}, \,\,\,\, \mu_i \in \RR.$$

Let $V \subsetneq X$ be the variety of $p$ in $X$ and let $X'=X \setminus V$. Let $p^\prime$ be the image of $p$ under the canonical projection from $\RR[X]$ to $\RR[X']=\RR[X]/I(X')$. Let $\ell' \in \RR[X']^*_{\leq 2d+2}$ be the linear functional given by
$$\ell'=\sum_{v_i \in X}\mu_i p(v_i)\ell_{v_i}=\sum_{v_i \in X'}\mu_ip'(v_i)\ell_{v_i}.$$
%Note that $\ell'$ is well-defined since only points in $X'$ occur in the above summation with a non-zero coefficient. 
%By Theorem \ref{THM pos} there exists $h$ in the interior of $\Sigma_{2d}$ such that $p' \cdot h \in \Sigma_{2d}(X')$. 
We claim that $\ell'$ separates $p'\Sigma(X')_{\leq 2d}$ from $\Sigma(X')_{\leq 2d+2}$. Indeed for any $q \in \Sigma(X)_{\leq 2d}$ we have $$\ell'(p'q)=\ell(p^2q) \geq 0,$$ 
while for any $q \in \Sigma(X')_{\leq 2d+2}$ we have $$\ell'(q)=\ell(pq)\leq 0.$$
By Theorem \ref{THM pos} it follows that $\ell'$ must be identically zero, which implies that $\ell$ is defined only in terms of evaluations on points of $V$. Thus $\ell$ vanishes identically on $p\Sigma(X)_{\leq 2d+2}$.

\end{proof}

\section{Lower Bound on Multipliers}\label{Section:LowerBound}
In this section we prove the lower bound on the degree of rational function representations for polynomials on the hypercube. We deal with $S_n$-invariant polynomials which vanish on a {\em level} $T = \{x \in C: \sum x_i = t\}$ of the hypercube $C=\{0,1\}^n$. Such functions come up naturally in combinatorial optimization, where we are counting objects subject to some symmetric restrictions; see Section \ref{Subsection:MaxCut}. We will show that such functions do not have rational sums of squares representations with multipliers of low degree. 

It will simplify the notation to use subsets of $[n]$ as exponents: $x^{\{1,4\}} = x_1x_4$. The vector space $\RR[C]$ of functions on the hypercube has a basis $\{x^m: m \subseteq [n]\}$ of squarefree monomials. Thus we can write any function $f \in \RR[C]$ as 
$f = \sum_{m \subset [n]} c_mx^m$, and we have $\deg(f) = \max \{|m|: c_m \ne 0\}$. We define $\RR[C]_d$ to be the collection of homogeneous degree-$d$ functions, and $\RR[C]_{\le d}=\oplus_{i=0}^d \RR[C]_i$ the collection of functions of degree at most $d$.

We also need to discuss the notion of divisibility in a coordinate ring. For instance, we may have $f,g,h\in \RR[X]$ with $f=gh$ but $\deg(f) < \deg(g) + \deg(h)$; in the case of the hypercube, $x \cdot x = x$. To fix this, for $f,g \in \RR[X]$, we say that $g$  \textit{properly divides $f$} if there exists $h \in \RR[X]$ such that $f=gh$ and $\deg(f) = \deg(g) + \deg(h)$. %Again, if $I$ is clear we can drop ``mod $I$.'' 
We will also say that \textit{$g$ properly divides $f$ to order $m$} if $g^m$ properly divides $f$, but $g^{m+1}$ does not.

We note that the symmetric group $S_n$ acts on $\RR[C]$ by permuting the variables directly: $(123) x_1 = x_2$. To start, we decompose $\RR[C]$ into {\em irreducible $S_n$-modules}. We introduce the necessary background in representation theory of the symmetric group below. For further information see the introduction by Sagan \cite{Sagan}, whose notation we adopt here. 

\subsection{Representation theory of $S_n$}
A {\em partition} of a positive integer $n$ is an ordered tuple $(\lambda_1, \ldots, \lambda_k)$ of positive integers such that $\lambda_1 \ge \ldots \ge \lambda_k$, and $\lambda_1 + \ldots + \lambda_k = n$. Corresponding to each partition is its diagram, where we draw $k$ rows of boxes, with $\lambda_i$ boxes in the $i$th row. For example, the partition $(4,2)$ of $n=6$ has the following diagram: %(show diagram here).

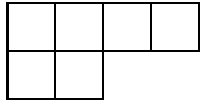
\ytableausetup{notabloids}
\begin{equation*}\ydiagram{4,2
}
\end{equation*}

A {\em tableau} of shape $\lambda$ is an assignment of numbers $\{1,\ldots,n\}$ to the boxes in the diagram of $\lambda$. %We will generally consider fillings using the numbers $\{1,\ldots,n\}$, but repeated numbers and gaps are also possible. 
%A {\em semistandard} tableau has strictly increasing columns and weakly increasing rows, and 
A {\em standard} tableau has strictly increasing rows and columns. Here is an example of a tableau and a standard tableau of shape $(4,2)$:

\begin{figure}[htd]
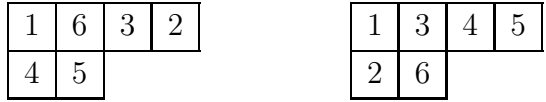

\label{Figure:TableauExample}
\ytableausetup{notabloids}
\begin{align*} \ytableaushort{
1632,45
} \hspace{2cm}
\ytableaushort{
1345,26
}\\
\end{align*}
\caption{A tableau and a standard tableau of shape $(4,2)$.}
\end{figure}

A {\em tabloid} is an equivalence class of tableaux, where we identify two tableaux if the fillings of their rows are the same as subsets of $\{1,\dots,n\}$. %For example, the following tabloid (show it) contains 24 tableaux, obtained by permuting the first row. Note that we write tabloids without bars in the rows to emphasize the equivalence relation, and we denote the tabloid containg a tableau $T$ by $[T]$. Also note that a given tabloid on numbers 1 through $n$ will contain either 0 or 1 standard tableaux. 

For a tableau $T$ and an element $\sigma \in S_n$, let $\sigma$ act on $T$ by permuting the numbers in $T$. Then the action of $S_n$ can be extended to tabloids and formal linear combination of tabloids. Formal linear combinations of tabloids of shape $\lambda$ form the \emph{permutation module} $M^\lambda$. %For example (show example). 

Let $C_T$ be the {\em column group} of $T$; that is, the subgroup of $S_n$ fixing the columns of $T$. Now we can define the {\em polytabloid} $e_T = \sum_{\sigma \in C_T} \textup{sign}(\sigma) \cdot [\sigma(T)]$, where $[\sigma(T)]$ is the tabloid equivalence class of $\sigma(T)$. %, defined to be the formal linear span of all tabloids of shape $\lambda$ filled with 1 through $n$. (Give example of one.) 
Now, define the {\em Specht module} $S^\lambda$: $$S^\lambda := \textup{span}( \{e_T: \hbox{$T$ is a standard tableau of shape $\lambda$}\}),$$ which is a submodule of $M^\lambda$. %This is the irreducible representation of $S_n$ corresponding to the partition $\lambda$. We have the following theorem.
%\begin{theorem}
%Up to isomorphism, the irreducible representations of $S_n$ are exactly the Specht modules $S^\lambda$ as $\lambda$ ranges over partitions of $n$.
%\end{theorem}
Irreducible representations (irreducible $S_n$-modules) of $S_n$ are precisely given by the Specht modules $S^\lambda$, where $\lambda$ is a partition of $n$.

\subsection{Functions on the hypercube $C$ and $S_n$-representations}
Recall that $S_n$ acts on $\RR[C]$ by permuting the variables. In the following we treat $(n,0)$ as an alias for the partition $(n)$ to simplify our notation.  We now define an isomorphism between tabloids and monomials. For  $k \le n/2$ let $M^{(n-k,k)}$ and $S^{(n-k,k)}$ denote the permutation and the Specht modules respectively, corresponding to the partition $(n-k,k)$. 

Define %\gb{need to define $M^{n-k,k}$} 
$\phi_k: M^{(n-k,k)} \to \RR[C]$ by $\phi_k([m^c,m]) = x^m$, and extend $\phi$ linearly. For example, $\phi_3([12345,678]) = x_6x_7x_8$. The image of $\phi_k$ is the subspace $\RR[C]_k$ of homogeneous functions of degree $k$. We also have $\RR[C]_k \cong \RR[C]_{n-k}$ as $S_n$-modules, since we can take complements in the exponent: if $n=6$, then $x_1x_2 \in \RR[C]_2 \leftrightarrow x_3x_4x_5x_6 \in \RR[C]_4$.%\jg{This paragraph is very dense, and mentions a lot of notions we never introduced and outsourced to Sagan. If we don't want to loose all the more optimization minded people, we might want to expand this to two or three paragraphs, and introduce the concepts more gently, even if we don't define them completely.}
\begin{proposition} \label{Proposition:CubeMultiplicities}
The $S_n$-module $\RR[C]$ decomposes into $n+1-2k$ copies of $S^{(n-k,k)}$, for $0 \le k \le \frac{n}{2}$.
\end{proposition}
\begin{proof}
By Young's rule (Theorem 2.11.2 in \cite{Sagan}), $M^{(n-k,k)}$ splits into direct sum of $S^{(n-i,i)}$ for $0 \le i \le k$, each coming with multiplicity $1$. By the above, if $k \le n/2$, $\RR[C]_{n-k} \cong \RR[C]_k \cong M^{(n-k,k)}$. 
If $n$ is odd, then 
\begin{align*}
\RR[C] &= \bigoplus_{0 \le k < n/2} (\RR[C]_k \oplus \RR[C]_{n-k} ) \\
&\cong 2 \bigoplus_{0 \le k < n/2}  M^{(n-k,k)} \\
&\cong 2 \bigoplus_{0 \le k < n/2}  \left ( \bigoplus_{i=0}^k S^{(n-i,i)} \right ) \\
&\cong 2  \bigoplus_{i=0}^{\lfloor n/2 \rfloor} \left(\frac{n-1}{2} -i +1 \right) S^{(n-i,i)},
\end{align*}
which gives the result. For even $n$ just add the single copy of $\RR[C]_{n/2} \cong M^{(n/2,n/2)}$.
\end{proof}

Proposition \ref{Proposition:CubeMultiplicities} gives the decomposition of $\RR[C]$ into irreducible submodules. To analyze a specific function $f \in \RR[C]$, we now give an explicit decomposition of $\RR[C]$. We choose a slightly idiosyncratic description which will be useful for our purposes. 
Fix $t \in \RR$ and let $\ell = t - \sum x_i$. Recalling that $S^{(n-k,k)} \subset M^{(n-k,k)}$, define $H_{k0} = \phi(S^{(n-k,k)}) \subseteq \RR[C]_k$. Since $\phi$ is an $S_n$-module isomorphism, we have $H_{k0} \cong S^{(n-k,k)}$. Then for $i = 1, \ldots, n-2k$, define $H_{ki} = (t-\sum_j x_j)^i \cdot H_{k0}$. Note that no element of $H_{k0}$ is properly divisible by $\ell$.

\begin{theorem} \label{Theorem:CubeDecomposition}
$\RR[C]$ has the following decomposition into irreducibles:
$$\RR[C] = \bigoplus_{k=0}^{\lfloor n/2 \rfloor} \left( \bigoplus_{i=0}^{n+1-2k}H_{ki}\right).$$ This decomposition respects degree: for any $d$,
$$\RR[C]_{\le d} = \bigoplus_{k+i \le d} H_{ki}.$$
\end{theorem}

\begin{proof}
By Proposition \ref{Proposition:CubeMultiplicities}, the above decomposition contains the correct number of each irreducible $S_n$-module. Therefore, it remains to show that the summands are linearly independent. 

By Corollary 2.11 in \cite{stanley}, the map $U: \RR[C]_k \to \RR[C]_{n-k}$ %\gb{bijection doesn't make sense} 
given by $U(f)=(\sum x_j)^{n-2k}f$ is a bijection. 
Therefore, the map $U': \RR[C]_k \to \RR[C]_{\le k + i}$ given by $f \mapsto  (t - \sum x_j)^{i}f$ is injective for $i \le n-2k$, %since $U'$ is a precomposition of $U$
by consideration of the top degree terms of $U'(f)$. 
Since $H_{ki} = U'(H_{k0})$, we have that $\deg(f) = k+i$ for each nonzero $f \in H_{ki}$; in particular, $H_{ki} \ne 0$. Since $S_n$ acts trivially on $(t - \sum_j x_j)^i$, we have $H_{ki} \cong H_{k0}$. By irreducibility, we know that vectors in $H_{ki}$ and $H_{k'i'}$ are linearly independent if $k \ne k'$. It remains to consider $H_{ki}$ for varying $i$; but since each nonzero $f_i \in H_{ki}$ has degree exactly $k+i$, these are linearly independent as well.

The expression for $\RR[C]_{\le d}$ now follows from the linear independence of the modules $H_{ki}$.
\end{proof}

We now show that proper divisibility holds for functions of low degree vanishing on a level $T$, i.e. on the subset of the hypercube where the sum of coordinates is equal to a fixed number $t$. %For the rest of this section, we abbreviate $\ell = t - \sum x_i,$ the parameter $t$ will be fixed.

\begin{lemma} \label{Lemma:Vanishing}
Let $T = \{x \in C: \sum_i x_i = t\}$, for fixed $t \in \{0,\ldots,n\}$. Suppose $f \in \RR[C]_{\le d}$, and $f$ vanishes on $T$. If $d \le t \le n-d$, then $f$ is properly divisible by $\ell=t-\sum x_i$.
\end{lemma}
\begin{proof}
Let  $V$ be the $S_n$-submodule of $\RR[C]_{\leq d}$  consisting of polynomials that are properly divisible by $\ell$ and let $$W = H_{00} \oplus \ldots \oplus H_{d0} \cong S^{(n)} \oplus \cdots \oplus S^{(n-d,d)}.$$ 
By Theorem \ref{Theorem:CubeDecomposition} we have $\RR[C]_{\leq d}=V\oplus W$.
Let $U \subset W$ be the $S_n$-submodule of polynomials vanishing on $T$. Since $W$ contains exactly one copy of each irreducible submodule of $\RR[C]_{\leq d}$ it suffices to show that $U=0$. Since the $H_{i0}$ are nonisomorphic irreducible $S_n$-modules, it follows that $$U=\bigoplus_{i \in I} H_{i0},$$ 
where $I$ is a subset of $\{0,\dots,d\}$. Now we claim that polynomials in $H_{i0}$ do not identically vanish on $T$ for all $0\leq i \leq d$. Since $H_{i0}$ is an irreducible $S_n$-module it suffices to exhibit a single polynomial $p \in H_{i0}$ not vanishing on $T$.

To see this, let $q$ be the standard tableau of shape $(n-i,i)$ where the first row contains $\{1,\dots,n-i\}$ and the second row contains $\{n-i+1,\dots,n\}$. Let $\hat x\in C$ be given by $$\hat x=e_{n-t+1}+\dots+e_n,$$
where $e_j$ denotes the $j$-th standard basis vector. 
%\jg{again, I don't think we should completely outsource all notation as I don't think  we ever introduced the notation $e_k$}\gb{these are just standard basis vectors, no? The other $e$ is the polytabloid, which is confusing...}
Since $i\leq t \leq n-i$, the support of $\hat x$ contains the second row of $q$ and does not contain any of the first $i$ entries of the first row of $q$. Consider $p=\phi(e_q)$, $p \in H_{i0}$, where $e_q$ is the polytabloid corresponding to $q$. It follows that $p(\hat x)=1$, since only the monomial $\phi(q)$ is nonzero on $\hat x$ in $\phi(e_q)$ and $\phi(q)(\hat x)=1$.
See Figure \ref{Figure:VanishingTableau} for an example.
\end{proof}

\begin{figure}[htd]
\label{Figure:VanishingTableau}
\ytableausetup{notabloids}
\begin{align*}q &= \ytableaushort{
1234567,89
}\\
\hat x & = (0,0,0,0,0,0,1,1,1) \\
p = \phi(e_q) &= x_8x_9 - x_1x_9 - x_8x_2 + x_1x_2
\end{align*}
\caption{A standard tableau $q$ with sorted rows, and the associated vector $\hat x$. Here $n=9,i=2,t=3$. We have $p(\hat x) = 1$.}
\end{figure}

Now we can prove our main result on lower bounds for the degree of denominators in $\RR[C]$. %\gb{fix $l$ versus $\ell$}

\begin{theorem}\label{Theorem:SosMultipliers}
Suppose $f\in \RR[C]_{\leq t}$ with $t \le n/2$ is an $S_n$-invariant polynomial and $f$ is properly divisible by $\ell=t-(x_1+\dots+x_n)$ to odd order. Then $f$ is not $(d_1,d_2)$-rsos for $d_1\leq \min\left\{\frac{n-\deg f}{2},t\right\}$, $d_2 \le t$.
\end{theorem}
\begin{proof}
Suppose that $f \sum g_i^2 = \sum h_j^2$ with $g_i\in \RR[C]_{\leq d_1}$, $g_i\neq 0$ and $h_j\in \RR[C]_{\leq d_2}$. Let $g=\sum g_i^2$ and $h=\sum h_j^2$. Without loss of generality we may assume that $g$ and $h$ are $S_n$-invariant polynomials, otherwise we may replace them by their $S_n$-symmetrizations.

Since $d_2 \leq t$ by Lemma \ref{Lemma:Vanishing} we can write $h_j=\ell^{a_j}q_j$ %with $S_n$-invariant polynomials $q_j$ 
with $ \deg q_j=\deg h_j- a_j$ and $q_j$ not vanishing on all of $T$. Therefore, after symmetrizing $h=\sum \ell^{2a_j}q_j^2$ we see that $h=\ell^{2a}q$ where $a=\min a_j$ and $q$ is an $S_n$-invariant polynomial, $\deg q=\deg h-2a$,  and $q$ is strictly positive on $T$.

Similarly, since $d_1 \leq t$ we argue that $g=\ell^{2b}r$, where $r$ is an $S_n$-invariant polynomial strictly positive on $T$, and $\deg r=\deg g -2b$. Finally, $f=\ell^c p$ where $c$ is odd and $p$ is an $S_n$-invariant polynomial not identically zero on $T$ 
%\jg{doesn't this imply stritly positive as well?} \gb{There is no SOS implication as for others, so I don't think so; at least not immediately...}
with $\deg p=\deg f-c$. Combining, we see that $$\ell^{2b+c}pr-\ell^{2a}q=0.$$
Let $\alpha=\min \{2a,2b+c\}$. By factoring out $\ell^\alpha$ in the equation above we obtain
$$\ell^\alpha s=0,$$
for an $S_n$-invariant polynomial $s \in \RR[C]$ of degree strictly less than $n$ since $d_1\leq \min\left\{\frac{n-t}{2},t\right\}$ and $d_2 \le t$. Since $q$ and $r$ are strictly positive on $T$ and $p$ is not identically zero on $T$%\jg{see note above}
, it follows that $s$ does not vanish on $T$. Thus $s$ is a non-zero symmetric polynomial in $\RR[C]$ vanishing on $C\setminus T$. Therefore $s=\beta \chi_T$ for some constant $\beta \ne 0$, where $\chi_T\in \RR[C]$ is the polynomial vanishing on $C\setminus T$ and equal to $1$ on $T$. However, it is not hard to check that $\deg \chi_T=n$ for any level $T$ and therefore we arrive at a contradiction.

\end{proof}

\begin{corollary}  \label{Corollary:Sos}
Fix $t \leq n/2$ and let $f \in \RR[C]_{\leq t}$ be an $S_n$-invariant polynomial. Suppose that $f$ is properly divisible by $\ell=t-(x_1+\dots+x_n)$ to odd order. Then $f$ is not $d$-sos for $d \le t$.
\end{corollary}
\begin{proof}
Apply Theorem \ref{Theorem:SosMultipliers} with $d_1=0$.
\end{proof}
Theorem \ref{Corollary:Quadratic} also follows immediately:

%\begin{corollary}\label{Corollary:Quadratic}
%Let $k=\lfloor \frac{n}{2}\rfloor$ and let $f\in \RR[C]$ be given by $$f=(x_1+\dots+x_n-k)(x_1+\dots+x_n-k-1).$$
%Then $f$ is nonnegative on $C$ but $f$ is not $(k-1,k)$-sos.
%\end{corollary}
\begin{proof}[Proof of Theorem \ref{Corollary:Quadratic}]
Apply Theorem \ref{Theorem:SosMultipliers}.
\end{proof}

% UPPER BOUND SECTION

% APPLICATIONS SECTION
\section{Applications}\label{Section:Applications}
We give two applications of our results. Section \ref{Subsection:MaxCut} deals with the MAXCUT problem on $K_n$, and is an application to combinatorial optimization. Section \ref{Subsection:Hilbert} deals with lower degree bounds in Hilbert's 17th problem.
% writing functions as $(d_1,d_2)$-sos on $\RR^n$, not mod $I$, and is an application to real algebraic geometry.

\subsection{The maxcut problem}\label{Subsection:MaxCut}
A {\em cut} in a graph arises from a partition of the vertices into two sets $S_1,S_2$, the cut being the collection of all edges from $S_1$ to $S_2$. Note that switching $S_1$ and $S_2$ gives the same cut. We write $C = [S_1,S_2] = [S_2,S_1]$, and let $|S| = $ the number of edges from $S_1$ to $S_2$. A {\em maximal cut} is a cut maximizing $|S|$.

 In the complete graph $K_n$, the maximal cuts come from any partition of $[n]$ into two sets of $n/2$ vertices when $n$ is even, or $(n\pm 1)/2$ when $n$ is odd. We note that a point $v\in C$ naturally defines a cut $S^v=[S_1,S_2]$ via $S_1=\{i \, \mid \, v_i=0\}$ and $S_2=\{i \, \mid \, v_i=1\}$. %Given a cut $S = [S_1,S_2]$ in $K_n$, let $x = x(S) \in \{0,1\}^n$ be defined by $x_i = 1$ if $i \in S_1$; $x_i = 0$ if $i \in S_2$. Observe that for $i \ne j$, the quantity 
%$$(2x_i-1)(2x_j-1) = \left\{
%\begin{array}{rl}
%1 & \textup{if $i,j$ are in the same half of $S$,} \\
%-1 & \textup{if $i,j$ are in different halves of $S$.}
%\end{array}
%\right.$$

Let $n$ be odd, Let $k=\lfloor \frac{n}{2}\rfloor$ and let $q\in \RR[C]$ be given by $$q=(x_1+\dots+x_n-k)(x_1+\dots+x_n-k-1).$$
We note that for all $v \in C$ we have $q(v)=|S^v|$.

%We will calculate the function on $C=\{0,1\}^n$ that counts the number of edges in a cut. 
%If $x$ represents a cut with $|S|$ edges, then 
%$$\sum_{i < j} (2x_i-1)(2x_j-1) = \left({n \choose 2} - |S|\right)\cdot 1 + \left(|S|\right) \cdot -1 = {n \choose 2} - 2|S|.$$
 %Therefore, $|S| = \frac{n^2-n}{4} - \frac{1}{2} \sum_{i < j} (2x_i-1)(2x_j-1)$. For $n$ odd, the maximum size of a cut is $\frac{n-1}{2}\frac{n+1}{2} = \frac{n^2-1}{4}$. Put 
%$$q = \frac{n^2-1}{4} - \left(\frac{n^2-n}{4} - \frac{1}{2} \sum_{i < j} (2x_i-1)(2x_j-1) \right).$$ Then we have the following factorization and inequality.

%\begin{proposition} \label{Proposition:CutFunction}
%Let $n$ be odd and $q$ as above. $q$ factors mod $I$ as $ (\frac{n-1}{2} - \sum x_i)(\frac{n+1}{2} - \sum x_i)$. For $x \in \{0,1\}^n$, $q(x) \ge 0$, and $q(x) = 0$ only for the vectors representing max cuts in $K_n$.
%\end{proposition}
%\begin{proof}
%The factorization is routine algebra; recall that the ideal $I$ of the cube allows simplification $x_i^2 = x_i$. The nonnegativity and equality statements follow from the preceding discussion.
%\end{proof}

Note that the $q$ defined above is the same polynomial as in Theorem \ref{Corollary:Quadratic}. %Therefore, it is not $d$-sos mod $I$ for $d \le t$. 
This allows us to reprove and strengthen a result of Laurent. In \cite{moniquestuff}, Theorem 4, it is shown that the Lasserre rank of the cut polytope of $K_n$, for $n$ odd, is at least $\frac{n+1}{2}$. %The Lasserre relaxation is a hierarchy of semidefinite program approximations to a 0/1 polytope based on sums of squares, and we can reprove Laurent's result by showing that $q$ is not $d$-sos for $d < \frac{n+1}{2}$.
This implies that there exists a quadratic polynomial $q \in \RR[C]_{\leq 2}$ such that $q$ is not $\frac{n-1}{2}$-sos. 
In fact the proof by Laurent established this for the same $q$ as above. However from Theorem \ref{Corollary:Quadratic} we know that in fact $q$ is not $\frac{n-1}{2}$-rsos. Further, it was conjectured in \cite{moniquestuff} that the Lasserre rank is precisely $\frac{n+1}{2}$ in this case. %\gb{Hmm, is this the conjecture or is conjecture on Lasserre rank?} 
This is equivalent to saying that any nonnegative quadratic $q \in \RR[C]_{\leq 2}$ that can be written as $q(x)=q_0 + \sum_{i \not = j} q_{ij}x_ix_j$ is $\frac{n+1}{2}$-sos.%\jg{this is now correct but maybe too much detail?} 
While we are not able to show this conjecture, it follows from Corollary \ref{Corollary:QuadraticBound} that any quadratic $q \in \RR[C]_{\leq 2}$ is $\frac{n+1}{2}$-rsos, and from Theorem \ref{Theorem:QuadraticBoundPositive} that even if we demand positive multipliers, $\frac{n+3}{2}$-rsos is enough.

%\begin{corollary}\label{Corollary:Cut}
%Let $n$ be odd and $I$ the ideal of the hypercube $C$. Then $q$, as above, is not a sum of squares mod $I$ of degree $\le \frac{n-1}{2}$.
%\end{corollary}
%\begin{proof}
%This is just Corollary \ref{Corollary:Quadratic}.
%\end{proof}

\subsection{Globally nonnegative function with large multipliers}\label{Subsection:Hilbert}

We finish with an application to Hilbert's 17th problem. %about sums of rational squares on $\RR^n$.

\begin{theorem} \label{Theorem:Global}
Let $k=\lfloor \frac{n}{2}\rfloor$. There exists a polynomial $p$ of degree $4$ nonnegative on $\RR^n$ which is not $k$-rsos in $\RR[x_1,\dots,x_n]$.
\end{theorem}
\begin{proof}
%Let $k=\lfloor \frac{n}{2}\rfloor$ and 
Let $f\in \RR[x_1,\dots,x_n]$ be given by $$f=(x_1+\dots+x_n-k)(x_1+\dots+x_n-k-1).$$
By Corollary \ref{Corollary:Quadratic} we know that $f$ is not $k$-rsos in $\RR[C]$. Using Lemma \ref{Lemma:GenClosed} with $X = C$ it follows that $f+\epsilon$ is not $k$-rsos in $\RR[C]$ for all sufficiently small $\epsilon>0$. Let $f'=f+\epsilon$ for a fixed such $\epsilon$.

Let $r=\sum_{i=1}^n (x_i^2-x_i)^2$. For sufficiently large $\lambda>0$ the polynomial $p=f'+\lambda r$ is strictly positive on $\RR^n$. 
Suppose that $p$ is $k$-rsos in $\RR[x_1,\dots,x_n]$: we have $ph=g$ with $(k-2)$-sos non-zero polynomial $h$, and $k$-sos polynomial $g$.

For $\alpha=(\alpha_1,\dots,\alpha_n) \in \RR^n$ let $C_{\alpha}$ be the hypercube given by equations $(x_i-\alpha_i)(x_i-\alpha_i-1)=0$. By Lemma \ref{Lemma:GenClosed} it follows that $p$ is not $(k-1,k)$-sos in $\RR[C_{\alpha}]$ for all $\alpha$ sufficiently close to $0$, since by linear change of variables it suffices to consider a small perturbation of $p$ in $\RR[C]$. However, there exist $\alpha$ arbitrarily close to $0$ such that $h \not \equiv 0$ in $\RR[C_{\alpha}]$. This is a contradiction since it follows that $p$ is $k$-rsos in $\RR[C_\alpha]$ for such $\alpha$.
\end{proof}

\bibliographystyle{plain}
\bibliography{main}

\begin{thebibliography}{10}

\bibitem{BPT}
Grigoriy Blekherman, Pablo~A. Parrilo, and Rekha~R. Thomas.
\newblock {\em Semidefinite optimization and convex algebraic geometry},
  volume~13 of {\em MOS-SIAM Series on Optimization}.
\newblock Society for Industrial and Applied Mathematics (SIAM), 2012.

\bibitem{CB}
David Eisenbud, Mark Green, and Joe Harris.
\newblock {Cayley-Bacharach} theorems and conjectures.
\newblock {\em Bulletin of the AMS}, 33(3):295--324, 1996.

\bibitem{goemans_williamson}
Michel Goemans and David Williamson.
\newblock Improved approximation algorithms for maximum cut and satisfiability
  problems using semidefinite programming.
\newblock {\em J. Assoc. Comput. Mach.}, 42(6):1115--1145, 1995.

\bibitem{GLPT}
Jo{\~a}o Gouveia, Monique Laurent, Pablo~A. Parrilo, and Rekha~R. Thomas.
\newblock A new semidefinite programming hierarchy for cycles in binary
  matroids and cuts in graphs.
\newblock {\em Math. Program.}, 133(1-2, Ser. A):203--225, 2012.

\bibitem{gpt}
Jo{\~a}o Gouveia, Pablo~A. Parrilo, and Rekha~R. Thomas.
\newblock Theta bodies for polynomial ideals.
\newblock {\em SIAM J. Optim.}, 20(4):2097--2118, 2010.

\bibitem{GKZ}
Feng Guo, Erich Kaltofen, and Lihong Zhi.
\newblock Certificates of impossibility of {Hilbert-Artin} representations of a
  given degree for definite polynomials and functions.
\newblock In {\em ISSAC '12 Proceedings of the 37th International Symposium on
  Symbolic and Algebraic Computation}, pages 195--202. ACM, 2012.

\bibitem{lasserre}
Jean~B. Lasserre.
\newblock Global optimization with polynomials and the problem of moments.
\newblock {\em SIAM J. Optim.}, 11(3):796--817, 2000/01.

\bibitem{lasserre2002}
Jean~B Lasserre.
\newblock An explicit equivalent positive semidefinite program for nonlinear
  0-1 programs.
\newblock {\em SIAM Journal on Optimization}, 12(3):756--769, 2002.

\bibitem{moniquestuff}
Monique Laurent.
\newblock Lower bound for the number of iterations in semidefinite hierarchies
  for the cut polytope.
\newblock {\em Math. Oper. Res.}, 28(4):871--883, 2003.

\bibitem{laurent2007}
Monique Laurent.
\newblock Semidefinite representations for finite varieties.
\newblock {\em Mathematical Programming}, 109(1):1--26, 2007.

\bibitem{Lombardi}
Henri Lombardi.
\newblock Une borne sur les degrŽs pour les thŽormes des zŽros rŽel effectif.
\newblock In {\em Real Algebraic Geometry, Proceedings, Rennes 1991}, volume
  1524 of {\em Lecture Notes in Mathematics}, pages 323--345. Springer-Verlag,
  1992.

\bibitem{LPR}
Henri Lombardi, Daniel Perrucci, and Marie-Fran\c{c}oise Roy.
\newblock Elementary recursive bounds for {Hilbert's} 17th problem.
\newblock {\em in preparation}.

\bibitem{marshall2008}
M.~Marshall.
\newblock {\em Positive Polynomials and Sums of Squares}.
\newblock Mathematical surveys and monographs. American Mathematical Society,
  2008.

\bibitem{parrilo2}
Pablo~A Parrilo.
\newblock Semidefinite programming relaxations for semialgebraic problems.
\newblock {\em Mathematical programming}, 96(2):293--320, 2003.

\bibitem{prestel2001}
A.~Prestel and C.N. Delzell.
\newblock {\em Positive Polynomials: From Hilbert's 17th Problem to Real
  Algebra}.
\newblock Springer Monographs in Mathematics. Springer, 2001.

\bibitem{Sagan}
Bruce~E. Sagan.
\newblock {\em The symmetric group}, volume 203 of {\em Graduate Texts in
  Mathematics}.
\newblock Springer-Verlag, New York, second edition, 2001.
\newblock Representations, combinatorial algorithms, and symmetric functions.

\bibitem{Schmid}
Joachim Schmid.
\newblock {\em On the degree complexity of Hilbert's 17th problem and the Real
  Nullstellensatz}.
\newblock Habilitation thesis, University of Dortmund, Germany, 1998.

\bibitem{stanley}
Richard~P. Stanley.
\newblock Variations on differential posets.
\newblock In {\em Invariant theory and tableaux ({M}inneapolis, {MN}, 1988)},
  volume~19 of {\em IMA Vol. Math. Appl.}, pages 145--165. Springer, New York,
  1990.

\end{thebibliography}
\end{document}